\documentclass[11pt]{article}
\usepackage[latin1]{inputenc}
\usepackage{amsmath,amsthm,amssymb}
\usepackage{amsfonts}
\usepackage{amsmath,amsthm,amssymb,amscd}
\usepackage{latexsym}
\usepackage{color}
\usepackage{graphicx}
\usepackage{enumitem}
\usepackage{mathrsfs}

\makeatletter \@addtoreset{equation}{section} \makeatother

\newtheorem{theorem}{Theorem}[section]
\newtheorem{definition}{Definition}[section]

\newtheorem{lemma}{Lemma}[section]
\newtheorem{remark}{Remark}[section]

\begin{document}	
	
		\title{\sc An elliptic problem  of the Prandtl-Batchelor \\ type with a singularity}
		\author{\sc Debajyoti Choudhuri$^{a}$
		and 
		Du\v{s}an D.  Repov\v{s}$^{b,}$\footnote{Corresponding author:  dusan.repovs@guest.arnes.si} \\
			\small{$^{a}$Department of Mathematics, National Institute of Technology Rourkela,}\\
			\small{  Rourkela, 769008, Odisha, India.	{\it Email: dc.iit12@gmail.com}}\\
			\small{$^b$Faculty of Education and Faculty of Mathematics and Physics, University of Ljubljana, and}\\
			\small{Institute of Mathematics, Physics and Mechanics, Ljubljana, 1000, Slovenia. {\it Email: dusan.repovs@guest.arnes.si}}\\	
			\\}			
		 \date{}
		 \maketitle
			\begin{abstract}
			\noindent  We establish the existence of at least two solutions of the {\it Prandtl--Batchelor} like elliptic problem driven by a power nonlinearity and a singular term. The associated energy functional is nondifferentiable and  hence the usual variational techniques do not work.  We shall  use a novel approach in tackling the associated energy functional by a sequence of $C^1$ functionals and a {\it cutoff function}. Our main tools are fundamental elliptic regularity theory and the mountain pass theorem.			
			\begin{flushleft}
				{\it Keywords}:~  Elliptic free boundary problems; Mountain Pass Theorem; singularity.\\
				{\it Math. Subj. Classif. (2020)}:~35R35, 35Q35, 35J20, 46E35.
			\end{flushleft}
		\end{abstract}
	\section{Introduction}\label{s1}
	 We consider the following class of sublinear elliptic {\it free boundary} problems
	\begin{align}\label{main_prob}
	\left\{\begin{aligned}
	\begin{split}
	-\Delta u=\alpha\chi_{\{u>1\}}(x)f(x,(u-1)_+)&+\beta u^{-\gamma}~\text{in}~\Omega\setminus G(u)\\
	|\nabla u^+|^2-|\nabla u^-|^2&=2~\text{on}~G(u)\\
	u&>0~\text{in}~\Omega\\
	u&=0~\text{on}~\partial\Omega.
	\end{split}
	\end{aligned}\right.
	\end{align}
	Here,
	 $\Omega\subset\mathbb{R}^N$ is a bounded domain, $N\geq 2$, $0<\gamma<1$, the boundary $\partial\Omega$ has $C^{2,a}$ regularity, $G(u)=\partial\{u:u>1\}$, $\alpha, \beta>0$ are parameters  and $\chi$ is an indicator function. Furthermore, $\nabla u^{\pm}$ are the limits of $\nabla u$ from the sets $\{u:u>1\}$ and $\{u:u\leq 1\}^{\circ}$ respectively, and $(u-1)_+=\max\{u-1,0\}$. The nonlinear term $f$ is a locally H\"{o}lder continuous function $f:\Omega\times\mathbb{R}\rightarrow[0,\infty)$ that satisfies the following conditions for all  $x\in\Omega, t>0$.
	\begin{eqnarray}
	(f_1)&~& \text{For some}~c_0, c_1>0, |f(x,t)|\leq c_0+c_1 t^{p-1},~\text{where}~1<p<2.\nonumber\\
	(f_2)&~& f(x,t)>0.
	\end{eqnarray}
	 We shall  prove the existence of two distinct nontrivial solutions of \eqref{main_prob} for a sufficiently large $\alpha$.\\
	The case when $f(x,t)=1$, $\beta=0$ is the well-known {\it Prandtl--Batchelor} problem, where the region $\{u:u>1\}$ represents the vortex patch bounded by the vortex line $\{u:u=1\}$ in a steady state fluid flow for $N=2$ (cf. {\sc Batchelor} \cite{Batchelor1, Batchelor2}). This case has been studied by several authors, e.g. {\sc Caflisch} \cite{Caflisch}, {\sc Elcrat and Miller} \cite{Elc_mil}, {\sc Acker} \cite{Acker},  and {\sc Jerison and Perera} \cite{Jeri_Perera}. We drew our motivation for studying the present problem in this paper from {\sc perera} \cite{Perera1}. The problem studied by {\sc Perera} \cite{Perera1} is the case when $\beta=0$ in problem \eqref{main_prob}.\\
	The nonlinearity $f$ includes the sublinear case of $f(x,t)=t^{p-1}$. {\sc Jerison and Perera} \cite{Jeri_Perera} considered problem \eqref{main_prob} with $\beta=0$ for $2<p<\infty$ if $N=2$, and $2<p\leq 2^*=\dfrac{2N}{N-2}$ if $N\ge 3$. This problem has its application in the study of plasma that is confined in a magnetic field. The region there $\{u:u>1\}$ represents the plasma and the boundary of the plasma is modelled by the free boundary (cf.  {\sc Caffarelli and Friedman} \cite{Caff_Fried1}, {\sc Friedman and Liu} \cite{Fried_Liu1}, and {\sc Temam} \cite{Temam1}). \\
	Elliptic problems driven by a singular term have, off late, been of great interest. However,  we shall  discuss only the seminal work of {\sc Lazer and McKenna} \cite{Lazer1} from 1991 that has opened a new door for the researchers in elliptic and parabolic PDEs. The problem considered in \cite{Lazer1} was as follows:
	\begin{align}\label{lazer}
	\left\{\begin{aligned}
	\begin{split}
	-\Delta u&=p(x)u^{-\gamma}~\text{in}~\Omega\\
	u&=0~\text{on}~\partial\Omega
	\end{split}
	\end{aligned}\right.
	\end{align}
	where $p>0$ is a $C^{a}(\bar{\Omega})$ function, $\gamma>0$, $\Omega$ is a bounded domain with a smooth boundary $\partial\Omega$ of $C^{2+a}$ regularity ($0<a<1$), and $N\geq 1$. The authors in \cite{Lazer1} proved that problem \eqref{lazer} has a unique solution $u\in C^{2,a}(\Omega)\cap C(\bar{\Omega})$ such that $u>0$ in $\Omega$. Another noteworthy work, addressing the singularity driven elliptic problem is due to {\sc Giacomoni et al.} \cite{Giaco1}. {\sc Jerison and Perera} \cite{Jeri_Perera} obtained a mountain pass solution of this problem for the superlinear subcritical case. {\sc Yang and Perera} \cite{Yang_Perera1} addressed the problem for the critical case. Recently, {\sc Choudhuri and Repov\v{s}} \cite{chou1} established the existence of a solution for a semilinear elliptic PDE with a free boundary condition on a stratified Lie group. Furthermore, those readers looking to expand their knowledge on the techniques and trends of the topics in analysis of elliptic PDEs may refer to {\sc Papageorgiou et al.} \cite{vetro1}.\\
	 We shall  prove that a solution of problem \eqref{main_prob} is Lipschitz continuous of class $H_0^1(\Omega)\cap C^2(\bar{\Omega}\setminus G(u))$ and is a classical solution on $\Omega\setminus G(u)$. This solution vanishes on $\partial \Omega$ continuously and satisfies the {\it free boundary} condition in the following sense:
	 \begin{align}\label{free_bdry_cond}
	 \underset{\epsilon^+\rightarrow 0}{\lim}\int_{\{u=1+\epsilon^+\}}(2-|\nabla u|^2)\psi\cdot\hat{n} dS-\underset{\epsilon^+\rightarrow 0}{\lim}\int_{\{u=1-\epsilon^+\}}|\nabla u|^2\psi\cdot\hat{n}dS&=0
	 \end{align}
	$  \ \hbox{for all} \  \psi\in C_0^1(\Omega,\mathbb{R}^N)$ that are supported a.e. on $\{u:u\neq 1\}$. Here $\hat{n}$ is the outward drawn normal to $\{u:1-\epsilon^-<u<1+\epsilon^+\}$ and $dS$ is the surface element.\\
	The novelty of this work, which separates it from the work of Perera \cite{Perera1}, lies in the efficient handling of the singular term that disallows the associated energy functional to be $C^1$ at $u=0$. This difficulty is the reason why one cannot directly apply the results from the variational set up. To handle this situation, we shall  define a {\it cut-off} function.
	\begin{remark}\label{norm}
	Note that $\int_{\Omega}|\nabla u|^2dx$ will be often denoted by $\|u\|^2$, where $\|\cdot\|$ is the norm of an element in the Sobolev space $H_0^1(\Omega)$. 
	\end{remark} 
We begin by defining a {\it weak solution} of problem \eqref{main_prob}.
\begin{definition}\label{weak_soln_defn}
A function $u\in H_0^1(\Omega)$, $u>0$ a.e. in $\Omega$ is said to be a weak solution of problem  \eqref{main_prob} if it satisfies the following:
\begin{align}\label{weak_soln}
0&=\int_{\Omega}\nabla u\cdot\nabla \varphi dx-\alpha\int_{\Omega} g(x,(u-1)_+))\varphi dx-\beta\int_{\Omega}u^{-\gamma}\varphi dx~~ \ \hbox{for all} \  ~\varphi\in H_0^1(\Omega).
\end{align}
\end{definition}
We define the associated energy functional to  problem \eqref{main_prob} as follows:
\begin{align}\label{energy_fnal}
E(u)&=\frac{1}{2}\|u\|^2+\int_{\Omega}(\chi_{\{u>1\}}(x)-\alpha G(x,(u-1)_+))dx-\frac{\beta}{1-\gamma}\int_{\Omega}(u^+)^{1-\gamma}dx~~ \ \hbox{for all} \  ~u\in H_0^1(\Omega),
\end{align}
where $F(x,t)=\int_0^tf(x,t)dt$, $t\geq 0$.\\
The functional $E$ fails to be of $C^1$ class due to the term $\int_{\Omega}(u^+)^{1-\gamma}dx$. Moreover, it is nondifferentiable due to the term $\int_{\Omega}\chi_{\{u>1\}}(x)dx$.  We shall  first tackle the singular term by defining a cut-off function $\phi_{\beta}$ as follows:
\[\phi_{\beta}(u)= \begin{cases}
u^{-\gamma}, & \text{if}~u>u_{\beta} \\
u_{\beta}^{-\gamma}, & \text{if}~u\leq u_{\beta}.
\end{cases}\]
Here $u_{\beta}$ is a solution of the following problem
\begin{align}\label{singular_prob}
\begin{split}
-\Delta u&=\beta u^{-\gamma},~\text{in}~\Omega\\
u&>0,~\text{in}~\Omega\\
u&=0,~\text{on}~\partial\Omega.
\end{split}
\end{align}
The existence of $u_{\beta}$ can be guaranteed by {\sc Lazer and McKenna} \cite{Lazer1}. Moreover, a solution of problem \eqref{singular_prob} is a {\it subsolution} of \eqref{main_prob} (refer to Lemma \ref{existence_positive_soln} in Section \ref{s6}). Note that,  we call \eqref{singular_prob} a {\it singular problem}.  We denote $\Phi_{\beta}(u)=\int_{0}^{u}\phi_{\beta}(t)dt$.\\
\noindent Furthermore, the functional $E$ is nondifferentiable and hence we approximate it by $C^1$ functionals. This technique is adopted from the work of {\sc Jerison and Pererra} \cite{Jeri_Perera}. Working along similar lines, we now define a smooth function $h:\mathbb{R}\rightarrow[0,2]$ as follows:

\[h(t)= \begin{cases}\label{smooth_func_1}
0, & \text{if}~t\leq 0 \\
\text{a positive function}, & \text{if}~0<t<1\\
0, & \text{if}~t\geq 1
\end{cases}\]
and $\int_0^1h(t)dt=1$.
We let $H(t)=\int_0^th(t)dt$. Clearly, $H$ is a smooth and nondecreasing function such that 
\[H(t)= \begin{cases}\label{smooth_func_2}
0, & \text{if}~t\leq 0 \\
\text{a positive function}<1, & \text{if}~0<t<1\\
1, & \text{if}~t\geq 1.
\end{cases}\]
We further define for $\delta>0$ 
\begin{align}\label{smooth_func_3}
f_{\delta}(x,t)=H\left(\frac{t}{\delta}\right)f(x,t),~F_{\delta}(x,t)=\int_0^tf_{\delta}(x,t)dt,~\text{for all}~t\geq 0.
\end{align}
Define,
\begin{align}\label{delta_energy}
E_{\delta}(u)&=\frac{1}{2}\|u\|^2+\int_{\Omega}\left[H\left(\frac{u-1}{\delta}\right)-\alpha F_{\delta}(x,(u-1)_+)-\beta\Phi_{\beta}(u)\right]dx,~\text{for all}~u\in H_0^1(\Omega).
\end{align}
The functional $E_{\delta}$ is of $C^1$ class. 
The main result of this paper is the following theorem.
\begin{theorem}\label{main_result}
Let conditions $(f_1)-(f_2)$ hold. Then there exist $\Lambda, \beta_*>0$ such that for all $\alpha>\Lambda$, $0<\beta<\beta_*$ problem \eqref{main_prob} has two {\it Lipschitz continuous} solutions, say $u_1,u_2\in H_0^1(\Omega)\cap C^2(\bar{\Omega}\setminus G(u))$, satisfying \eqref{main_prob} classically in $\bar{\Omega}\setminus G(u)$. These solutions also satisfy the free boundary condition in the generalized sense and vanish continuously on $\partial\Omega$. Furthermore,
\begin{enumerate}
\item $E(u_1)<-|\Omega|\leq -|\{u:u=1\}|<E(u_2)$, where $|\cdot|$ denotes the Lebesgue measure in $\mathbb{R}^N$, hence $u_1, u_2$ are nontrivial solutions.
\item $0<u_2\leq u_1$ and the regions $\{u_1:u_1<1\}\subset\{u_2:u_2<1\}$ are connected where $\partial\Omega$ is connected. The sets $\{u_2>1\}\subset\{u_1>1\}$ are nonempty.  
\item $u_1$ is a minimizer of $E$ (but $u_2$ is not).
\end{enumerate}
\end{theorem}

The paper is organized as follows.
In Section \ref{s2}  we  introduce the key preliminary facts.
In Section \ref{s3}  we  prove a convergence lemma.
In Section \ref{s4}  we  prove a free boundary condition.
In Section \ref{s5}  we  prove two auxiliary lemmas.
In Section \ref{s6} we prove a result on positive Radon measure.
Finally, in Section \ref{s7}  we prove the main theorem.

\section{Preliminaries}\label{s2}

\noindent An important result that will be used to pass to the limit in the proof of  Lemma \ref{convergence1} is the following theorem due to {\sc Caffarelli et al.} \cite[Theorem $5.1$]{Caffa_jeri_kenig}.
\begin{lemma}\label{conv_res1}
Let $u$ be a Lipschitz continuous function on the unit ball $B_1(0)\subset\mathbb{R}^N$ satisfying the distributional inequalities
$$\pm\Delta u\leq A\left(\dfrac{1}{\delta}\chi_{\{|u-1|<\delta\}}(x)H(|\nabla u|)+1\right)$$
for constants $A>0$, $0<\delta\leq 1$, $H$ is a continuous function obeying $H(t)=o(t^2)$ as $t\to\infty$. Then there exists a constant $C>0$ depending on $N, A$ and $\int_{{B_1}(0)}u^2dx$, but not on $\delta$, such that 
$$\underset{x\in B_{\frac{1}{2}}(0)}{\sup}|\nabla u(x)|\leq C.$$
\end{lemma}
\noindent The following are the {\it Palais--Smale} condition and the mountain pass Theorem.
\begin{definition}(cf. {\sc Kesavan} \cite[Definition $5.5.1$]{Kesh})\label{PS_COND}
Let $V$ be a Banach space and $J:V\rightarrow\mathbb{R}$ be a $C^1$-functional. Then $J$ is said to satisfy the Palais-Smale $(PS)$ condition if the following holds: Whenever $(u_n)$ is a sequence in $V$ such that $(J(u_n))$ is bounded and $(J'(u_n))\rightarrow 0$ strongly in $V^*$ (the dual space), then $(u_n)$ has a strongly convergent subsequence in $V$.
\end{definition}
\begin{lemma}\label{MP_THM}(cf. {\sc Alt and Caffarelli} \cite[Theorem $2.1$]{Ambrosetti})
Let $J$ be a $C^1$-functional defined on a Banach space $V$. Assume that $J$ satisfies the $(PS)$-condition and that there exists an open set $U\subset V$, $v_0\in U$, and $v_1\in X\setminus\bar{U}$ such that 
$$\underset{v\in\partial U}{\inf}J(v)>\max\{J(v_0),J(v_1)\}.$$
Then $J$ has a critical point at the level 
$$c=\underset{\psi\in\Gamma}\inf~\underset{u\in \psi([0,1])}\max{J(v)}\geq \underset{u\in\partial U}\inf J(u),$$
where $\Gamma=\{\psi\in C([0,1]):\psi(0)=v_0,\psi(1)=v_1\}$ is the class of paths in $V$ joining $v_0$ and $v_1$.
\end{lemma}
\noindent Before we prove Lemma  \ref{convergence1}, we would like to give an a priori estimate of the parameter $\beta$.
\section{Convergence lemma}\label{s3}
 We shall  denote the first eigenvalue of $(-\Delta)$ by $\alpha_1$ and the first eigenvector by $\varphi_1 $ (for an existence of $\alpha_1$,  $\varphi_1$ refer to {\sc Kesavan} \cite{Kesh}). Fix $\alpha$ to, say, $\alpha_0$ and let $\beta$ be any positive real number. On testing problem \eqref{main_prob} with $\varphi_1$ the following weak formulation has to hold if $u$ is a weak solution of problem \eqref{main_prob}. Thus
\begin{align}\label{beta_est}
\begin{split}
\alpha_1\int_{\Omega}u\varphi_1dx&=\int_{\Omega}\nabla u\cdot\nabla \varphi dx=\alpha\int_{\Omega}f(x,(u-1)_+)\varphi_1dx+\beta\int_{\Omega}(u^+)^{-\gamma}\varphi dx.
\end{split}
\end{align}
So there exists $\beta_{*}>0$ which depends on the chosen fixed $\alpha$, such that $\beta_*t^{-\gamma}+\alpha f(x,(t-1)_+)>\alpha_1 t$ for all $t>0$. This is a contradiction to \eqref{beta_est}. Therefore, $0<\beta<\beta_*$.
\begin{lemma}\label{convergence1}
Let conditions $(f_1)-(f_2)$ hold, $\delta_j\rightarrow 0$ ($\delta_j>0$) as $j\rightarrow\infty$, and let $u_j$ be a critical point of $E_{\delta_j}$. If $(u_j)$ is bounded in $H_0^1(\Omega)\cap L^{\infty}(\Omega)$, then there exists a Lipschitz continuous function $u$ on $\bar{\Omega}$ such that $u\in H_0^1(\Omega)\cap C^2(\bar{\Omega}\setminus G(u))$ and a subsequence such that
\begin{enumerate}[label=(\roman*)]
\item $u_j\rightarrow u$ uniformly over $\bar{\Omega}$,
\item $u_j\rightarrow u$ locally in $C^1(\bar{\Omega}\setminus\{u=1\})$,
\item $u_j\rightarrow u$ strongly in $H_0^1(\Omega)$,
\item $E(u)\leq\lim\inf E_{\delta_j}(u_j)\leq\lim\sup E_{\delta_j}(u_j)\leq E(u)+|\{u:u=1\}|$, i.e. $u$ is a nontrivial function if $\lim\inf E_{\delta_j}(u_j)<0$ or $\lim\sup E_{\delta_j}(u_j)>0$.
\end{enumerate}
Furthermore, $u$ satisfies $$-\Delta u=\alpha\chi_{\{u>1\}}(x)g(x,(u-1)_+)+\beta u^{-\gamma}$$ classically in $\Omega\setminus G(u)$, the free boundary condition is satisfied in the generalized sense and $u$ vanishes continuously on $\partial\Omega$. If $u$ is nontrivial, then $u>0$ in $\Omega$, the region $\{u:u<1\}$ is connected, and the region $\{u:u>1\}$ is nonempty.
\end{lemma}
\begin{proof}[Proof of Lemma \ref{convergence1}]
Let $0<\delta_j<1$. Consider the following problem:
\begin{align}\label{approx_prob_1}
\left\{\begin{aligned}
\begin{split}
-\Delta u_j&=-\frac{1}{\delta_j}h\left(\frac{u_j-1}{\delta_j}\right)+\alpha f_{\delta_j}(x,(u_j-1)_+)+\beta \phi_{\beta}(u_j)~\text{in}~\Omega\\
u_j&>0~\text{in}~\Omega\\
u_j&=0~\text{on}~\partial\Omega.
\end{split}
\end{aligned}\right.
\end{align}
The nature of the problem being a sublinear one and driven by a singularity allows us to conclude by an iterative technique that the sequence $(u_j)$ is bounded in $L^{\infty}(\Omega)$. Therefore, there exists $C_0$ such that $0\leq f_{\delta_{j}}(x,(u_j-1)_+)\leq C_0$. Let $\varphi_0$ be a solution of the following problem
\begin{align}\label{approx_prob_2}
\left\{\begin{aligned}
\begin{split}
-\Delta \varphi_0&=\alpha C_0+\beta u_{\beta}^{-\gamma}~\text{in}~\Omega,\\
\varphi_0&=0~\text{on}~\partial\Omega.
\end{split}
\end{aligned}\right.
\end{align}
Now since $h\geq 0$, we have that $-\Delta u_j\leq \alpha C_0+\beta u_{\beta}^{-\gamma}=-\Delta\varphi_0$ in $\Omega$. Therefore by the maximum principle, 
\begin{align}\label{comparison1}
0\leq u_j(x)\leq \varphi_0(x),~ \ \hbox{for all} \   x\in\Omega.
\end{align}
From the argument used in the proof of Lemma \ref{auxprob_appendix}, together with $\beta_*>0$ and large $\Lambda>0$, we conclude that $u_j>u_{\beta}$ in $\Omega$ for all $\beta\in(0,\beta_*)$. Since $\{u_j:u_j\geq 1\}\subset\{\varphi_0:\varphi_0\geq 1\}$, hence $\varphi_0$ gives a uniform lower bound, say $d_0$, on the distance from the set $\{u_j:u_j\geq 1\}$ to $\partial\Omega$. Furthermore, $u_j$ is a positive function satisfying the singular problem in a $d_0$-neighborhood of $\partial\Omega$. Thus $(u_j)$ is bounded with respect to the $C^{2,a}$ norm. Therefore, it has a convergent subsequence in the $C^2$-norm in a $\dfrac{d_0}{2}$ neighborhood of the boundary $\partial\Omega$. Obviously, $0\leq h\leq 2\chi_{(-1,1)}$ and hence
\begin{align}\label{comparison2}
\begin{split}
\pm\Delta u_j&=\pm\frac{1}{\delta_j}h\left(\frac{u_j-1}{\delta_j}\right)\mp\alpha f_{\delta_j}(x,(u_j-1)_+)+\beta u_j^{-\gamma}\\
&\leq\frac{2}{\delta_j}\chi_{\{|u_j-1|<\delta_j\}}(x)+\alpha C_0+\beta u_j^{-\gamma}\\
&\leq\frac{2}{\delta_j}\chi_{\{|u_j-1|<\delta_j\}}(x)+\alpha C_0+\beta u_{\beta}^{-\gamma}.
\end{split}
\end{align}
By {\sc Lazer and McKenna} \cite{Lazer1}, for any subset $K$ of $\Omega$ that is relatively compact in it, i.e. $K\Subset\Omega$, we have that $u_{\beta}\geq C_K$ for some $C_K>0$. Therefore
\begin{align}\label{comparison3}
\begin{split}
\pm\Delta u_j&\leq\frac{2}{\delta_j}\chi_{\{|u_j-1|<\delta_j\}}(x)+\alpha C_0+\beta C_K^{-\gamma}.
\end{split}
\end{align}
Since, $(u_j)$ is bounded in $L^2(\Omega)$ and by Lemma \ref{conv_res1} it follows that there exists $A>0$ such that 
\begin{align}\label{aux1}
\underset{x\in B_{\frac{r}{2}}(x_0)}{\sup}|\nabla u_j(x)|&\leq\frac{A}{r}
\end{align}
for suitable $r>0$ such that $B_r(0)\subset\Omega$. Therefore, $(u_j)$ is uniformly Lipschitz continuous on the compact subsets of $\Omega$ such that its distance from the boundary $\partial\Omega$ is at least $\frac{d_0}{2}$ units.\\
Thus by the {\it Ascoli--Arzela} theorem applied to $(u_j)$, we have a subsequence, still denoted the same, such that it converges uniformly to a Lipschitz continuous function $u$ in $\Omega$ with zero boundary values and with strong convergence in $C^2$ on a $\frac{d_0}{2}$-neighborhood of $\partial\Omega$. By the {\it Eberlein--\v{S}mulian} theorem we can conclude that $u_j\rightharpoonup u$ in $H_0^1(\Omega)$.\\
We now prove that $u$ satisfies the following equation
$$-\Delta u=\alpha\chi_{\{u>1\}}(x)f(x,(u-1)_+)+\beta u^{-\gamma}$$
in the set $\{u\neq 1\}$. This will include the cases $(i)~0<u_{\beta}<1<u$, $(ii)~1<u_{\beta}<u$, $(iii)~0<u_{\beta}<u<1$.  The cases $(i)-(iii)$ do not pose any real mathematical obstacle. Let $\varphi\in C_0^{\infty}(\{u>1\})$. Then $u\geq 1+2\delta$ on the support of $\varphi$ for some $\delta>0$. Using the convergence of $u_j$ to $u$ uniformly on $\Omega$, we have $|u_j-u|<\delta$ for any sufficiently large $j,\delta_j<\delta$. So $u_j\geq 1+\delta_j$ on the support of $\varphi$. Testing \eqref{convergence1} with $\varphi$ yields
\begin{align}\label{weak_conv_1}
\int_{\Omega}\nabla u_j\cdot\nabla\varphi dx&=\alpha\int_{\Omega}f(x,u_j-1)\varphi dx+\beta\int_{\Omega} u_j^{-\gamma}\varphi dx.
\end{align}
On passing to the limit $j\rightarrow\infty$ we get 
\begin{align}\label{weak_conv_2}
\int_{\Omega}\nabla u\cdot\nabla\varphi dx&=\alpha\int_{\Omega}f(x,u-1)\varphi dx+\beta\int_{\Omega} u^{-\gamma}\varphi dx.
\end{align}
To arrive at \eqref{weak_conv_2}, we have used the weak convergence of $u_j$ to $u$ in $H_0^1(\Omega)$ and the uniform convergence of the same in $\Omega$. Hence $u$ is a weak solution of $-\Delta u=\alpha f(x,u-1)+\beta u^{-\gamma}$ in $\{u>1\}$. Since $f, u$ are continuous and  Lipschitz continuous respectively, we conclude by the Schauder estimates that it is also a classical solution of $-\Delta u=\alpha f(x,u-1)+\beta u^{-\gamma}$ in $\{u:u>1\}$. Similarly, on choosing $\varphi\in C_0^{\infty}(\{u:u<1\})$, one can find a $\delta>0$ such that $u\leq 1-2\delta$. Therefore, $u_j<1-\delta$. Using the arguments as in \eqref{weak_conv_1} and \eqref{weak_conv_2}, we find that $u$ satisfies 
$-\Delta u=\beta u^{-\gamma}$ in the set $\{u:u<1\}$.\\
\noindent Let us now see what is the nature of $u$ in the set $\{u:u\leq 1\}^{\circ}$. On testing \eqref{convergence1} with any nonnegative function, passing to the limit $j\rightarrow\infty$, and using the fact that $h\geq 0$, $H\leq 1$, we can show that $u$ satisfies 
\begin{align}\label{weak_conv_3}
-\Delta u&\leq\alpha f(x,(u-1)_+)+\beta u^{-\gamma}~\text{in}~\Omega
\end{align}
in the distributional sense. Also, we see that $u$ satisfies $-\Delta u=\beta u^{-\gamma}$ in the set $\{u:u<1\}$ . Furthermore, $\mu=\Delta u+\beta u^{-\gamma}$ is a positive Radon measure supported on $\Omega\cap\partial\{u:u<1\}$ (refer to Lemma \ref{positive_Rad_meas} in Section \ref{s6}). From \eqref{weak_conv_3}, the positivity of the Radon measure $\mu$ and the usage of Section $9.4$ in {\sc Gilbarg and Trudinger} \cite{Gil_trud}, we conclude that $u\in W_{\text{loc}}^{2,p}(\{u:u\leq 1\}^{\circ})$, $1<p<\infty$. Thus $\mu$ is supported on $\Omega\cap\partial\{u:u<1\}\cap\partial\{u:u>1\}$ and $u$ satisfies $-\Delta u=\beta u^{-\gamma}$ in the set $\{u:u\leq 1\}^{\circ}$.\\
\noindent To prove $(ii)$,  we show that $u_j\rightarrow u$ locally in $C^1(\Omega\setminus\{u:u=1\})$. Note that we have already proved that $u_j\rightarrow u$ in the $C^2$ norm in a neighborhood of $\partial\Omega$ of $\bar{\Omega}$. Suppose that $M\subset\subset\{u:u>1\}$. In this set $M$ we have $u\geq 1+2\delta$ for some $\delta>0$. Thus, for sufficiently large $j$ with $\delta_j<\delta$, we have $|u_j-u|<\delta$ in $\Omega$ and hence $u_j\geq 1+\delta_j$ in $M$. From \eqref{approx_prob_1} we derive that $$-\Delta u=\alpha f(x,u-1)+\beta u^{-\gamma}~\text{in}~M.$$
Clearly, $f(x,u_j-1)\rightarrow f(x,u-1)$ in $L^p(\Omega)$ for $1<p<\infty$ because $f$ is a locally H\"{o}lder continuous function and $u_j\rightarrow u$ uniformly in $\Omega$. Our analysis says something stronger. Since $-\Delta u=\alpha f(x,u-1)$ in $M$, we have that $u_j\rightarrow u$ in $W^{2,p}(M)$. By the embedding $W^{2,p}(M)\hookrightarrow C^1(M)$ for $p>2$, we have that $u_j\rightarrow u$ in $C^1(M)$. This shows that $u_j\rightarrow u$ in $C^1(\{u>1\})$. Working along similar lines we can also show that $u_j\rightarrow u$ in $C^1(\{u:u<1\})$.\\
\noindent  We shall  now prove $(iii)$. Since $u_j\rightharpoonup u$ in $H_0^1(\Omega)$, we know that by the weak lower semicontinuity of the norm $\|\cdot\|$, 
$$\|u\|\leq\lim\inf\|u_j\|.$$
It suffices to prove that $\lim\sup\|u_j\|\leq \|u\|$. To achieve this, we multiply \eqref{approx_prob_1} with $u_j-1$ and then integrate by parts.  We shall  also use the fact that $th\left(\frac{t}{\delta_j}\right)\geq 0$ for any $t$. This gives
\begin{align}\label{weak_conv_4}
\begin{split}
\int_{\Omega}|\nabla u_j|^2dx&\leq \alpha\int_{\Omega}f(x,(u_j-1)_+)(u_j-1)_+dx-\int_{\partial\Omega}\frac{\partial u_j}{\partial\hat{n}}dS+\beta\int_{\Omega}u_j^{-\gamma}(u_j-1)_+dx\\
&\rightarrow\alpha\int_{\Omega}f(x,(u-1)_+)(u-1)_+dx-\int_{\partial\Omega}\frac{\partial u}{\partial\hat{n}}dS+\beta\int_{\Omega}u^{-\gamma}(u-1)_+dx
\end{split}
\end{align}
as $j\rightarrow\infty$. Here, $\hat{n}$ is the outward drawn normal to $\partial\Omega$. We saw earlier that $u$ is a weak solution to $-\Delta u=\alpha f(x,u-1)+\beta u^{-\gamma}$ in $\{u:u>1\}$. Let $0<\delta<1$. We test this equation with the function $\varphi=(u-1-\delta)_+$ and get
\begin{align}\label{weak_conv_5}
\int_{\{u>1+\delta\}}|\nabla u|^2dx&=\alpha\int_{\Omega}f(x,(u-1)_+)(u-1-\delta)dx+\beta\int_{\Omega}u^{-\gamma}(u-1-\delta)_+dx.
\end{align}
Integrating $(u-1-\delta)_{-}\Delta u=\beta u^{-\gamma}(u-1-\delta)_{-}$ over $\Omega$ yields
\begin{align}\label{weak_conv_6}
\int_{u<1-\delta}|\nabla u|^2dx&=-(1-\delta)\int_{\partial\Omega}\frac{\partial u}{\partial\hat{n}}dS+\beta\int_{\Omega}u^{-\gamma}(u-1-\delta)_{-}dx.
\end{align}
On adding \eqref{weak_conv_5} and \eqref{weak_conv_6} and passing to the limit $\delta\rightarrow 0$, we get
\begin{align}\label{weak_conv_7}
\int_{\Omega}|\nabla u|^2dx&=\alpha\int_{\Omega}f(x,(u-1)_+)(u-1)_+dx-\int_{\partial\Omega}\frac{\partial u}{\partial\hat{n}}dS+\beta\int_{\Omega}u^{-\gamma}(u-1)_+dx.
\end{align}
Note that we have used $\int_{\{u:u=1\}}|\nabla u|^2dx=0$. Invoking \eqref{weak_conv_7} and \eqref{weak_conv_4}, we get 
\begin{align}\label{weak_conv_8}
\lim\sup\int_{\Omega}|\nabla u_j|^2dx&\leq\int_{\Omega}|\nabla u|^2dx.
\end{align}
This proves $(iii)$.\\
\noindent  We shall  now prove $(iv)$. Consider
\begin{align}\label{weak_conv_9}
\begin{split}
E_{\delta_j}(u_j)=&\int_{\Omega}\left(\frac{1}{2}|\nabla u_j|^2+H\left(\frac{u_j-1}{\delta_j}\right)\chi_{\{u\neq 1}\}-\alpha F_{\delta_j}(x,(u_j-1)_+)-\beta u_j^{-\gamma}(u_j-1)_+\right)dx\\
&+\int_{\{u=1\}}H\left(\frac{u_j-1}{\delta_j}\right)dx.
\end{split}
\end{align} 
Since $u_j\rightarrow u$ in $H_0^1(\Omega)$ and $H\left(\frac{u_j-1}{\delta_j}\right)\chi_{\{u\neq 1\}}$, $F_{\delta_j}(x,(u_j-1)_+)$ are bounded and converge pointwise to $\chi_{\{u:u>1\}}$ and $F(x,(u-1)_+)$, respectively, it follows that the first integral in \eqref{weak_conv_9} converges to $E(u)$. Moreover, 
$$0\leq \int_{\{u:u=1\}}H\left(\frac{u_j-1}{\delta_j}\right)dx\leq |\{u:u=1\}|.$$
This proves $(iv)$.\qed
\section{Free boundary condition}\label{s4}
We shall  now show that $u$ satisfies the free boundary condition in the generalized sense (refer to condition \eqref{free_bdry_cond}). We choose $\vec{\varphi}\in C_0^1(\Omega,\mathbb{R}^N)$ such that $u\neq 1$ a.e. on the support of $\vec{\varphi}$. multiplying $\nabla u_j\cdot\vec{\varphi}$ to \eqref{approx_prob_1} and integrating over the set $\{u:1-\epsilon^-<u<1+\epsilon^+\}$ gives
\begin{align}\label{weak_conv_10}
\begin{split}
\int_{\{u:1-\epsilon^-<u<1+\epsilon^+\}}\left[-\Delta u_j+\frac{1}{\delta_j}h\left(\frac{u_j-1}{\delta_j}\right)\right]\nabla u_j\cdot\vec{\varphi} dx\\
=\int_{\{u:1-\epsilon^-<u<1+\epsilon^+\}}(\alpha f_{\delta_j}(x,(u_j-1)_+)+\beta u_j^{-\gamma})\nabla u_j\cdot\vec{\varphi} dx.
\end{split}
\end{align}
The term on the left-hand side of \eqref{weak_conv_10} can be expressed as follows:
\begin{align}\label{weak_conv_11}
\nabla\cdot\left(\frac{1}{2}|\nabla u_j|^2\vec{\varphi}-(\nabla u_j\cdot\vec{\varphi})\nabla u_j\right)+\nabla u_j\cdot (\nabla\vec{\varphi}\cdot\nabla u_j)-\frac{1}{2}|\nabla u_j|^2\nabla\cdot\vec{\varphi}+\nabla H\left(\frac{u_j-1}{\delta_j}\right)\cdot\vec{\varphi}.
\end{align}
Using this, we integrate by parts to obtain
\begin{align}\label{weak_conv_12}
\begin{split}
\int_{\{u:u=1+\epsilon^+\}\cup\{u=1-\epsilon^-\}}\left[\frac{1}{2}|\nabla u_j|^2\vec{\varphi}-(\nabla u_j\cdot\vec{\varphi})\nabla u_j+H\left(\frac{u_j-1}{\delta_j}\vec{\varphi}\right)\right]\cdot\hat{n}dx\\
=\int_{\{u:1-\epsilon^-<u<1+\epsilon^+\}}\left(\frac{1}{2}|\nabla u_j|^2\vec{\varphi}-(\nabla u_j\cdot\vec{\varphi})\nabla u_j\right)dx\\
+\int_{\{u:1-\epsilon^-<u<1+\epsilon^+\}}\left[H\left(\frac{u_j-1}{\delta_j}\right)\nabla\cdot\vec{\varphi}+\alpha f_{\delta_j}(x,(u_j-1)_+)\nabla u_j\cdot\vec{\varphi}+\beta u_j^{-\gamma}\nabla u_j\cdot\vec{\varphi}\right]dx.
\end{split}
\end{align}
By using $(ii)$, the integral on the left of equation \eqref{weak_conv_12} converges to 
\begin{align}\label{weak_conv_13}
\int_{\{u:u=1+\epsilon^+\}\cup\{u=1-\epsilon^-\}}\left(\frac{1}{2}|\nabla u|^2\varphi-(\nabla u\cdot\vec{\varphi})\nabla u\right)\cdot\hat{n}dS+\int_{\{u:u=1+\epsilon^+\}}\vec{\varphi}\cdot\hat{n}dS. 
\end{align}
Equation \eqref{weak_conv_13} is further equal to 
\begin{align}\label{weak_conv_14}
\int_{\{u:u=1+\epsilon^+\}}\left(1-\frac{1}{2}|\nabla u|^2\right)\vec{\varphi}\cdot\hat{n}dS-\int_{\{u:u=1-\epsilon^-\}}\frac{1}{2}|\nabla u|^2\vec{\varphi}\cdot\hat{n}dS.
\end{align}
This is because $\hat{n}=\pm\dfrac{\nabla u}{|\nabla u|}$ on the set $\{u:u=1+\epsilon^{\pm}\}\cup\{u:u=1-\epsilon^{\pm}\}$. By using $(iii)$ the first integral on the right-hand side of \eqref{weak_conv_12} converges to 
\begin{align}\label{weak_conv_15}
\int_{\{u:1-\epsilon^-<u<1+\epsilon^+\}}\left(\frac{1}{2}|\nabla u|^2\nabla\cdot\vec{\varphi}-\nabla u D\vec{\varphi}\cdot\nabla u\right)dx
\end{align}
whereas the second integral of \eqref{weak_conv_12} is bounded by 
\begin{align}\label{weak_conv_16}
\int_{\{u:1-\epsilon^-<u<1+\epsilon^+\}}(|\nabla\cdot\vec{\varphi}|+C|\vec{\varphi}|)dx
\end{align}
for some constant $C>0$. The last two integrals \eqref{weak_conv_15}-\eqref{weak_conv_16} vanish as $\epsilon^{\pm}\rightarrow 0$ since\\ $|\text{supp}(\vec{\varphi})\cap\{u:u=1\}|=0$. Therefore we first let $j\rightarrow\infty$ and then we let $\epsilon^{\pm}\rightarrow 0$ in \eqref{weak_conv_12} to prove that $u$ satisfies the free boundary condition.
\end{proof}
\noindent Using $(f_1)$ 
\begin{align}\label{weak_conv_17}
E_{\delta}(u)&\geq \int_{\Omega}\left\{\frac{1}{2}|\nabla u|^2-\alpha\left(c_0(u-1)_{+}+\frac{c_1}{p}(u-1)_+^p\right)-\frac{\beta}{1-\gamma} u^{1-\gamma}\right\}dx.
\end{align}
Clearly, since $1<p<2$, we have that $E_{\delta}$ is bounded from below and coercive. Thus $E_{\delta}$ satisfies the $(PS)$ condition (see Definition \ref{PS_COND}). It is easy to see that every $(PS)$ sequence is bounded by coercivity and hence contains a convergent subsequence by a standard argument--we extract weakly convergent subsequence and show that this weak limit is the strong limit of possibly, a different subsequence. Let us show that $E_{\delta}$ has a minimizer, say, $u_1^{\delta}$. By $(f_2)$, we have $F(x,t)>0$ for all $x\in\Omega$ and $t>0$. Thus for any $u\in H_0^1(\Omega)$ with $u>1$ on a set of positive measure, we have 
\begin{align}\label{weak_conv_18}
\int_{\Omega}F(x,(u-1)_+)dx&>0.
\end{align}
Therefore, $E(u)\rightarrow-\infty$ as $\alpha\rightarrow\infty$. Thus, there exists $\Lambda>0$ such that for all $\alpha>\Lambda$  we have
\begin{align}\label{weak_conv_19}
m_1(\alpha)&=\underset{u\in H_0^1(\Omega)}{\inf}\{E(u)\}<-|\Omega|.
\end{align}
Set $$\delta_0(\alpha)=\min\left\{\frac{|m_1(\alpha)|}{2\alpha c_0|\Omega|},\left(\frac{pc_0}{c_1}\right)^{\frac{1}{p-1}}\right\}.$$
\section{Auxiliary lemmas}\label{s5}
 We shall now establish the existence of the first solution of problem \eqref{main_prob} which also is a minimizer for the functional $E$. Let us begin with the following lemma.
\begin{lemma}\label{aux_lemma_1}
For all $\alpha>\Lambda$, $0<\beta<\beta_*$, $\delta<\delta_0(\alpha)$, the functional $E_{\delta}$ has a minimizer $u_1^{\delta}>0$ that satisfies
\begin{align}\label{weak_conv_20}
E_{\delta}(u_1^{\delta})&\leq m_1(\alpha)+2\alpha\delta c_0|\Omega|<0.
\end{align}
\begin{proof}
Since $E_{\delta}$ is bounded below and satisfies the $(PS)$ condition, it possesses a minimizer $u_1^{\delta}$. Also since $H\left(\frac{t-1}{\delta}\right)\leq \chi_{(1,\infty)}(t)$ for all $t$, we have
\begin{align}\label{weak_conv_21}
\begin{split}
E_{\delta}(u)-E(u)&\leq \alpha\int_{\Omega}[F(x,(u-1)_+)-F_{\delta}(x,(u-1)_+)]dx\\
&=\alpha\int_{\Omega}\int_{0}^{(u-1)_+}\left[1-H\left(\frac{t}{\delta}f(x,t)\right)\right]dtdx\\
&\leq\alpha\int_{\Omega}\int_0^{\delta}f(x,t)dtdx\\
&\leq\alpha\left(c_0\delta+\frac{c_1}{p}\delta^p\right)|\Omega|~\text{by}~(f_1).
\end{split}
\end{align}
Further, for $\delta<\delta_0(\alpha)$ we obtain \eqref{weak_conv_20}. Since, $E_{\delta}(u_1^{\delta})<0=E_{\delta}(0)$, this implies that $u_1^{\delta}$ is a nontrivial solution of problem \eqref{approx_prob_1}. This solution is positive since it is a minimizer.
\end{proof}
\end{lemma}
\noindent  We shall  now prove that the functional $E_{\delta}$ has a second nontrivial critical point, say $u_2^{\delta}$. 
\begin{lemma}\label{app_second_soln}
For any $\alpha>\Lambda$ and $0<\beta<\beta_*$, there exists a constant $c_3(\alpha)$ such that for all $\delta<\delta_0(\alpha)$, the functional $E_{\delta}$ has a second critical point $0<u_2^{\delta}\leq u_1^{\delta}$ that obeys
$$c_3(\alpha)\leq E_{\delta}(u_2^{\delta})\leq \frac{1}{2}\|u_1^{\delta}\|^2+|\Omega|.$$
Furthermore, $\emptyset\neq \{u_2^{\delta}:u_2^{\delta}>1\}\subset\{u_1^{\delta}:u_1^{\delta}>1\}$.
\end{lemma}
\begin{proof}
Choose some $\delta<\delta_0(\alpha)$. Consider
$$h_{\delta}(x,t)=\frac{1}{\delta}h\left(\frac{\min\{t,u_1^{\delta}(x)\}-1}{\delta}\right),~~H_{\delta}(x,t)=\int_0^th_{\delta}(x,t)dt,$$
$$\tilde{f}_{\delta}(x,t)=f_{\delta}(x,(\min\{t,u_1^{\delta}(x)\}-1)_+),~~\tilde{F}_{\delta}(x,t)=\int_0^t\tilde{f}_{\delta}(x,t)dt.$$
Further, we set 
$$\tilde{E}_{\delta}(u)=\int_{\Omega}\left[\frac{1}{2}|\nabla u|^2+H_{\delta}(x,u)-\alpha \tilde{F}_{\delta}(x,u)-\beta \phi_{\beta}(u)\right]dx,~~u\in H_0^1(\Omega).$$
The functional $\tilde{E}_{\delta}$ is of $C^1$ class and its critical points coincide with the weak solutions of the following problem
\begin{align}\label{app_prob_3}
\begin{split}
\left\{\begin{aligned}
-\Delta u&=-h_{\delta}(x,u)+\alpha \tilde{f}_{\delta}(x,u)+\beta\phi_{\beta}(u)~\text{in}~\Omega\\
u&=0~\text{on}~\partial\Omega.
\end{aligned}\right.
\end{split}
\end{align}
By the elliptic (Schauder) regularity, a weak solution of \eqref{app_prob_3} is also a classical solution. Also, by the maximum principle, we have that $u\leq u_1^{\delta}$. Thus $u$ is a weak solution of problem \eqref{approx_prob_2} and hence is a critical point of $\tilde{E}_{\delta}$, with $\tilde{E}_{\delta}(u)=E_{\delta}(u)$.  We shall now show that $\tilde{E}_{\delta}$ has a critical point, say $u_2^{\delta}$, that satisfies
\begin{align}\label{weak_conv_22}
m_2(\alpha)\leq \tilde{E}_{\delta}(u_2^{\delta})\leq\frac{1}{2}\|u_1^{\delta}\|^2+|\Omega|,~\text{for some}~m_2(\alpha)>0. 
\end{align}
This enables us to conclude that $E_{\delta}(u_2^{\delta})=\tilde{E}_{\delta}(u_2^{\delta})>0>E_{\delta}(u_1^{\delta})$ which in turn will imply that $u_2^{\delta}>0$ and different from $u_1^{\delta}$.\\
By the mountain pass Theorem see Lemma \ref{MP_THM}), the functional $\tilde{E_{\delta}}$ that is coercive (owing to its sublinear nature), satisfies the $(PS)$ condition. Clearly, for any $t\leq 1$, we have $$\tilde{f}_{\delta}(x,t)=f_{\delta}(x,0)$$ 
and 
$$\tilde{f}_{\delta}(x,t)\leq c_0+c_1(\min\{t,u_1^{\delta}(x)\}-1)_+^{p-1}\leq c_0+c_1(t-1)^{p-1}~\text{for}~t>1.$$
By $(f_1)$, we get
$$\tilde{F}_{\delta}(x,t)\leq c_0(t-1)_++\frac{c_1}{p}(t-1)_+^p\leq \left(c_0+\frac{c_1}{p}\right)|t|^q$$
for all $t$ with $q>2$ if $N=2$ and $2<q\leq \frac{2N}{N-2}$ if $N\geq 3$. We observe that
\begin{align}\label{weak_conv_23}
\tilde{E}_{\delta}(u)&\geq \int_{\Omega}\left[\frac{1}{2}|\nabla u|^2-\alpha \left(c_0+\frac{c_1}{p}\right)|u|^q-\beta |u|^{1-\gamma}\right]dx\\
&\geq\frac{1}{2}\|u\|^2-\lambda c_4 \left(c_0+\frac{c_1}{p}\right)\|u\|^q-\beta c_5\|u\|^{1-\gamma}.
\end{align}
By the embedding result $H_0^1(\Omega)\hookrightarrow L^q(\Omega)$ for $q>2$, the integral in \eqref{weak_conv_23} is positive if $\|u\|=r$, i.e., when $u\in\partial B_r(0)$, for sufficiently small $r>0$ where $B_r(0)=\{u\in H_0^1(\Omega):\|u\|<r\}$. Furthermore, since $\tilde{E}_{\delta}(u_1^{\delta})=E_{\delta}(u_1^{\delta})<0=\tilde{E}_{\delta}(0)$, we choose $r<\|u_1^{\delta}\|$, and then applying the mountain pass Theorem (Lemma \ref{MP_THM}), we get a critical point $u_2^{\delta}$ of $\tilde{E}_{\delta}$ with
$$\tilde{E}_{\delta}(u_2^{\delta})=\underset{\psi\in\Gamma}{\inf}\underset{u\in\psi([0,1])}{\max}\tilde{E}_{\delta}(u)\geq m_2(\alpha),$$
where $\Gamma=\{\psi\in C([0,1],H_0^1(\Omega)):\psi(0)=0,\psi(1)=u_1^{\delta}\}$ is the class of paths joining $0$ and $u_1^{\delta}$. For the path $\psi_0(t)=tu_1^{\delta}$, $t\in[0,1]$, we have
\begin{align}\label{weak_conv_24}\tilde{E}_{\delta}(\psi_0(t))&\leq\int_{\Omega}\left(\frac{1}{2}|\nabla u_1^{\delta}|^2+H_{\delta}(x,u_1^{\delta})\right)dx\end{align}
since $H_{\delta}(x,t)$ is nondecreasing in $t$ and $\tilde{F}_{\delta}(x,t)\geq 0$ for all $t$ by condition $(f_2)$. Since 
\begin{align}\label{weak_conv_25}H_{\delta}(x,u_1^{\delta}(x))=\int_{0}^{u_1^{\delta}}\frac{1}{\delta}h\left(\frac{t-1}{\delta}\right)dt=H\left(\frac{u_1^{\delta}(x)-1}{\delta}\right)\leq 1,\end{align}
it follows by \eqref{weak_conv_24} and \eqref{weak_conv_25} that
\begin{align}\label{weak_conv_26}
\begin{split}
\tilde{E}_{\delta}(u_2^{\delta})&\leq\underset{u\in\psi_0([0,1])}{\max}\tilde{E}_{\delta}(u)\leq\int_{\Omega}\left(\frac{1}{2}|\nabla u_1^{\delta}|^2+1\right)dx\\
&=\frac{1}{2}\|u_1^{\delta}\|^2+|\Omega|.
\end{split}
\end{align}
\end{proof}

\section{Positive Radon measure}\label{s6}

We shall now prove two more results that will be needed in the last section.
\begin{lemma}
	\label{existence_positive_soln}
	Let $0<\beta<\beta_*$. Then a solution of the problem
	\begin{align}\label{auxprob_appendix}
	\begin{split}
	\left\{\begin{aligned}
	-\Delta v&=\beta v^{-\gamma}~\text{in}~\Omega,\\
	v&>0~\text{in}~\Omega,\\
	v&=0~\text{on}~\partial\Omega,
	\end{aligned}\right.
	\end{split}
	\end{align}
	say $u_{\beta}$, satisfies ${u}_{\beta}<u$ a.e. in $\Omega$, where $u$ be a solution of problem \eqref{main_prob}.
\end{lemma}
\begin{proof}
	Let $u\in H_0^1(\Omega)$ be a positive solution of problem \eqref{main_prob} and $u_{\beta}>0$ a solution of problem \eqref{auxprob_appendix}. For any $0<\beta<\beta_*$, define a weak solution $u_{\beta}$ of problem \eqref{auxprob_appendix} as follows:
	\begin{align}\label{auxprob_appendix_WF}
	\begin{split}
	0&=\int_{\Omega}\nabla u_{\beta}\cdot\nabla \varphi dx-\beta\int_{\Omega}u_{\beta}^{-\gamma}\varphi dx,~~ \ \hbox{for all} \  ~\varphi\in H_0^1(\Omega).
	\end{split}
	\end{align}
	 By the {\it Schauder estimates} we have $u\in C^{2,a}(\Omega)$, and by {\sc Lazer and McKenna} \cite{Lazer1} we have $u_{\beta}\in C^{2,a}(\Omega)\cap C(\bar{\Omega})$.  We shall  show that $u\geq {u}_{\beta}$ a.e. in $\Omega$. We let $\tilde{\Omega}=\{x\in\Omega:u(x)<{u}_{\beta}(x)\}$. Thus, from the weak formulations satisfied by $u$, ${u}_{\beta}$ and testing with the function $\varphi=(u_{\beta}-u)_+$, we have 
	\begin{align}\label{comp_1}
	\begin{split}
	0\leq\int_{\Omega}\nabla(u_{\beta}-u)\cdot\nabla(u_{\beta}-u)_+dx=&-\alpha\int_{\Omega}\chi_{\{u>1\}}g(x,(u-1)_+)(u_{\beta}-u)_+dx\\
	&+\beta\int_{\Omega}(u_{\beta}^{-\gamma}-u^{-\gamma})(u_{\beta}-u)_+dx\leq 0
	\end{split}
	\end{align}
	Thus, $\|(u_{\beta}-u)_+\|=0$ and hence $|\tilde{\Omega}|=0$. However, since the functions $u, u_{\beta}$ are continuous, it follows that $\tilde{\Omega}=\emptyset$.
	Hence, by \eqref{comp_1}, we obtain $u\geq\underline{u}_{\beta}$ in $\Omega$.\\	
	Let $W=\{x\in\Omega:u(x)=u_{\beta}(x)\}$. Since $W$ is a measurable set, it follows that for any $\eta>0$ there exists a closed subset $V$ of $W$ such that $|W\setminus V|<\eta$. Further assume that $|W|>0$. We now define a test function $\varphi\in C_c^1(\mathbb{R}^N)$ such that 
	\begin{equation}\varphi(x)=\begin{cases}
	1& ~\text{if}~ x\in V,\\
	0<\varphi<1&~\text{if}~x\in W\setminus V,\\
	0&~\text{if}~x\in \Omega\setminus W.
	\end{cases}\end{equation}
	Since $u$ is a weak solution to \eqref{main_prob}, we have 
	\begin{align}\label{equality_breakage}
	\begin{split}
	0=&\int_{\Omega}-\Delta u\varphi dx-\beta\int_{V}u^{-\gamma}dx-\beta\int_{W\setminus V}u^{-\gamma}\varphi dx-\int_{V}f(x,(u-1)_+)dx-\int_{W\setminus V}f(x,(u-1)_+)\varphi dx\\
	=&-\int_{V}f(x,(u-1)_+)dx-\int_{W\setminus V}f(x,(u-1)_+)\varphi dx<0.
	\end{split}
	\end{align}
	This is a contradiction. Therefore, $|W|=0$ which implies that $W=\emptyset$. Hence, $u>u_{\beta}$ in $\Omega$. 
\end{proof}
\begin{lemma}\label{positive_Rad_meas}
Function $u$ is in $H_{\text{loc}}^{1,2}(\Omega)$ and the Radon measure $\mu=\Delta u+\beta u^{-\gamma}$ is nonnegative and supported on $\Omega\cap\{u:u<1\}$ for $\beta\in(0,\beta_*)$.
\end{lemma}
\begin{proof}
We follow the proof due to {\sc Alt and Caffarelli} \cite{alt_caffa}. Choose $\delta>0$, $\beta\in(0,\beta_*)$, and a test function $\varphi^2\chi_{\{u:u<1-\delta\}}$ where $\varphi\in C_0^{\infty}(\Omega)$. Therefore,
\begin{align}\label{app_2}
\begin{split}
0&=-\int_{\Omega}\nabla u\cdot\nabla(\varphi^2\min\{u-1+\delta,0\})dx+\beta\int_{\Omega}u^{-\gamma}\varphi^2\min\{u-1+\delta,0\}dx\\
&=\int_{\Omega\cap\{u:u<1-\delta\}}\nabla u\cdot\nabla(\varphi^2(u-1+\delta))dx+\beta\int_{\Omega\cap\{u:u<1-\delta\}}u^{-\gamma}(\varphi^2(u-1+\delta))dx\\
&=\int_{\Omega\cap\{u:u<1-\delta\}}|\nabla u|^2\varphi^2dx+2\int_{\Omega\cap\{u:u<1-\delta\}}\varphi\nabla u\cdot\nabla\varphi(u-1+\delta)dx+\beta\int_{\Omega\cap\{u:u<1-\delta\}}u^{-\gamma}(\varphi^2(u-1+\delta))dx.
\end{split}
\end{align}
By an application of integration by parts to the second term of \eqref{app_2}, we get 
\begin{align}
\begin{split}
\int_{\Omega\cap\{u:u<1-\delta\}}|\nabla u|^2\varphi^2dx&=-2\int_{\Omega\cap\{u:u<1-\delta\}}\varphi\nabla u\cdot\nabla\varphi(u-1+\delta)dx+\beta\int_{\Omega\cap\{u:u<1-\delta\}}u^{-\gamma}(\varphi^2(u-1+\delta))dx\\
&\leq 4\int_{\Omega}u^2|\nabla\varphi|^2dx-\beta\int_{\Omega}u^{1-\gamma}\varphi^2dx\leq  4\int_{\Omega}u^2|\nabla\varphi|^2dx.
\end{split}
\end{align}
On passing to the limit $\delta\rightarrow 0$, we conclude that $u\in H_{\text{loc}}^{1,2}(\Omega)$.\\
\noindent Furthermore, for nonnegative $\zeta\in C_0^{\infty}(\Omega)$ we have
\begin{align}\label{app_3}
\begin{split}
-\int_{\Omega}\nabla\zeta\cdot\nabla udx+\beta\int_{\Omega}u^{-\gamma}\zeta dx=&\left(\int_{\Omega\cap\{u:0<u<1-2\delta\}}+\int_{\Omega\cap\{u:1-2\delta<u<1-\epsilon\}}+\int_{\Omega\cap\{u:1-\delta<u<1\}}\right.\\
&\left.+\int_{\Omega\cap\{u:u>1\}}\right)\\
&\times 
\left[\nabla\left(\zeta\max\left\{\min\left\{2-\frac{1-u}{\delta},1\right\},0\right\}\right)\cdot\nabla u\right.\left.+\beta u^{-\gamma}\zeta\right]dx\\
\geq& \int_{\Omega\cap\{u:1-2\delta<u<1-\delta\}}
\left[\left(2-\frac{1-u}{\delta}\right)\nabla\zeta\cdot\nabla u+\frac{\zeta}{\delta}|\nabla u|^2\right.\left.+\beta u^{-\gamma}\zeta\right]dx.
\end{split}
\end{align}
On passing to the limit $\delta\rightarrow 0$, we obtain $\Delta(u-1)_{-}\geq 0$ in the distributional sense, and hence there exists a Radon measure $\mu$ (say) such that $\mu=\Delta(u-1)_{-}\geq 0$.
\end{proof}
\section{Proof of the main theorem}\label{s7}
Finally we are in a position to prove  Theorem \ref{main_result}.
\begin{proof}[Proof of Theorem \ref{main_result}]
Choose $\alpha>\lambda$ and a sequence $\delta_j\rightarrow 0$ such that $\delta_j<\delta_0(\alpha)$. For each $j$, Lemma \ref{aux_lemma_1} gives a minimizer $u_1^{\delta}>0$ of $E_{\delta_j}$ that obeys
\begin{align}\label{weak_conv_27}
E_{\delta_{j}}(u_1^{\delta_j})&\leq m_1(\alpha)+2\alpha\delta_j c_0|\Omega|<0.
\end{align}
Further, by Lemma \ref{app_second_soln}, we can guarantee the existence of the second critical point $0<u_2^{\delta}\leq u_1^{\delta_j}$ such that 
\begin{align}\label{weak_conv_28}
m_2(\alpha)&\leq E_{\delta_j}(u_2^{\delta_j})\leq\frac{1}{2}\|u_1^{\delta_j}\|^2+|\Omega|.
\end{align}
The next step is to show that $(u_1^{\delta_j})$, $(u_2^{\delta_j})$ are bounded in $H_0^1(\Omega)\cap L^{\infty}(\Omega)$.  We shall  then apply Lemma \ref{convergence1}.\\
Since $H\geq 0$ and 
$$H_{\delta}(x,(t-1)_+)\leq c_0(t-1)_++\frac{c_1}{p}(t-1)_+^p\leq \left(c_0+\frac{c_1}{p}\right)|t|^p$$
for all $t$ by $(f_1)$, it follows that
\begin{align}\label{weak_conv_29}
\frac{1}{2}\|u_1^{\delta}\|^2&\leq E_{\delta}(u_1^{\delta})+\alpha\left(c_0+\frac{c_1}{p}\right)\int_{\Omega}(u_1^{\delta})^pdx+\beta \int_{\Omega}(u_1^{\delta})^{1-\gamma}dx.
\end{align}
Since $E_{\delta_j}(u_1^{\delta})<0$ by \eqref{weak_conv_27} and $p<2$, we have that $(u_1^{\delta_j})$ is bounded in $H_0^1(\Omega)$.\\
\noindent Since $f_{\delta}(x,(t-1)_+)=f_{\delta}(x,0)=0$ for any $t\leq 1$ and 
$$f_{\delta}(x,(t-1)_+)\leq c_0+c_1(t-1)^{p-1}\leq (c_0+c_1)t^{p-1}$$
whenever $t>1$ by $(f_1)$, we get
\begin{align}\label{weak_conv_30}
-\Delta u_1^{\delta_j}&=-\frac{1}{\delta_j}h\left(\frac{u_1^{\delta_j}-1}{\delta_j}\right)+\alpha f_{\delta_j}(x,(u_1^{\delta_j}-1)_+)+\beta (u_1^{\delta_j})^{-\gamma}\leq\alpha(c_0+c_1)(u_1^{\delta_j})^{p-1}+\beta (u_1^{\delta_j})^{-\gamma}.
\end{align}
However, when $u_1^{\delta_j}<1$, 
\begin{align}\label{weak_conv_31}
-\Delta u_1^{\delta_j}&=\beta (u_1^{\delta_j})^{-\gamma}
\end{align}
in which case $u_1^{\delta_j}=u_{\beta}|_{\{u_1^{\delta_j}<1\}}$.
\end{proof}
The sublinearity of \eqref{weak_conv_31} together with the boundedness of $(u_1^{\delta_j})$ in $H_0^1(\Omega)$ implies by the {\it Moser} iteration method that $(u_1^{\delta_j})$ in $L^{\infty}(\Omega)$. By a similar argument, $(u_2^{\delta_j})$ is also bounded in $L^{\infty}(\Omega)$ since $0<u_1^{\delta_j}\leq u_2^{\delta_j}$ in $\Omega$. On renaming the subsequence of $(\delta_j)$, the sequences $(u_1^{\delta_j})$, $(u_2^{\delta_j})$ converge uniformly to a Lipschitz continuous functions, say $u_1,u_2\in H_0^1(\Omega)\cap C^2(\bar{\Omega}\setminus G(u))$ respectively, of problem \eqref{main_prob} that satisfies
$$-\Delta u=\alpha\chi_{\{u>1\}}f(x,(u-1)_+)+\beta u^{-\gamma}$$
classically in the region $\Omega\setminus G(u)$, the free boundary condition in the generalized sense and furthermore, continuously vanishes on $\partial\Omega$. We also have that 
\begin{align}\label{weak_conv_32}
E(u_1)\leq\lim\inf E_{\delta_j}(u_1^{\delta_j})\leq \lim\sup E_{\delta_j}(u_1^{\delta_j})\leq E(u_1)+|\{u_1:u_1=1\}|
\end{align}
and
\begin{align}\label{weak_conv_33}
E(u_2)\leq\lim\inf E_{\delta_j}(u_2^{\delta_j})\leq \lim\sup E_{\delta_j}(u_2^{\delta_j})\leq E(u_2)+|\{u_2:u_2=1\}|.
\end{align}
Using \eqref{weak_conv_32} in combination with \eqref{weak_conv_27} and \eqref{weak_conv_19}, yields 
$$E(u_1)\leq \lim\sup E_{\delta_j}(u_1^{\delta_j})\leq m_1(\alpha)\leq E(u_1).$$
Therefore, 
\begin{align}\label{weak_conv_34}
E(u_1)&=m_1(\alpha)<-|\Omega|.
\end{align}
Similarly, combining \eqref{weak_conv_33} with \eqref{weak_conv_28} yields 
$$0<m_2(\alpha)\leq\lim\inf E_{\delta_j}(u_2^{\delta_j})\leq E(u_2)+|\{u_2:u_2=1\}|.$$
Thus, 
\begin{align}\label{weak_conv_35}
E(u_2)>-|\{u_2:u_2=1\}|\geq-|\Omega|.
\end{align}
So, from \eqref{weak_conv_34} and \eqref{weak_conv_35} we can conclude that $u_1, u_2$ are distinct and nontrivial solutions of problem \eqref{main_prob}. Here $u_1$ is a minimizer whereas $u_2$ is not . Also, since $u_2^{\delta_j}\leq u_1^{\delta_j}$ for each $j$, we have $u_2\leq u_1$. Since $u_2$ is a nontrivial solution, it follows that $0<u_2\leq u_1$ and the sets $\{u_1:u_1<1\}\subset\{u_2:u_2<1\}$ are connected if $\partial\Omega$ is connected. Moreover, the sets $\{u_2:u_2>1\}\subset\{u_1:u_1>1\}$ are nonempty. \qed 
 
\subsection*{Funding}
The first author (DC) has received funding from the National Board for Higher Mathematics (NBHM), Department of Atomic Energy (DAE) India, [02011/47/2021/NBHM(R.P.)/R\&D II/2615].
The second author (DDR) has received funding from the Slovenian Research Agency grants P1-0292, J1-4031, J1-4001, N1-0278, N1-0114, and N1-0083.

	\end{document}